\documentclass[12pt, reqno]{amsart}
\usepackage{helvet}
 
\usepackage[english]{babel}
\usepackage{float,subfloat}
\usepackage{here}
\usepackage[letterpaper,top=2cm,bottom=2cm,left=3cm,right=3cm,marginparwidth=1.75cm]{geometry}
\usepackage{graphicx, subcaption}
\usepackage[colorlinks=true, allcolors=blue]{hyperref}
\usepackage{amsmath, amsthm, mathtools, amscd, amsfonts, amssymb}
\usepackage{xcolor}
\definecolor{darkgreen}{rgb}{0,0.51,0.11}
\usepackage{lmodern}
\usepackage{enumitem}
\usepackage{tikz-cd}
\newtheorem{theorem}{Theorem}[section]
\newtheorem{lemma}[theorem]{Lemma}
\newtheorem{proposition}[theorem]{Proposition}
\newtheorem{corollary}[theorem]{Corollary}
\theoremstyle{definition}
\newtheorem{definition}[theorem]{Definition}
\newtheorem{example}[theorem]{Example}

\newtheorem{remark}[theorem]{Remark}
\numberwithin{equation}{section}

\newcommand{\Z}{\ensuremath{\mathbb{Z}}}

\title{ On a bi-lateral Adding Machine and its characterization}

\author[]{P. Mehdipour}
	\address[Pouya Mehdipour]{%
		Departamento de Matem\'atica, Universidade Federal de Vi\c cosa, Brazil.
	}
	\email{pouya@ufv.br}
\author[]{R. M. S. dos Santos}
	\address[Rebeca M. S. dos Santos]{%
		Departamento de Matem\'atica, Universidade Federal de Vi\c cosa, Brazil.
	}
	\email{rebeca.marjorie@ufv.br}



\begin{document}
\maketitle

\begin{abstract}
In this paper, we introduce a bilateral adding machine based on a zip space with two sets of alphabets. We demonstrate that these adding machines are homeomorphisms and provide necessary and sufficient conditions for their characterization.
\end{abstract} 
\section{Introduction}

The \emph{adding machine}, also known as the \emph{odometer}, is a classical example in the study of topological dynamics and symbolic systems. It is defined as a minimal and equicontinuous dynamical system on a compact zero-dimensional space. Formally, an adding machine associated to a sequence of integers $\alpha = (j_1, j_2, \dots)$, with each $j_i \geq 2$, is the shift map on the inverse limit space $\Delta_\alpha = \varprojlim \mathbb{Z}/p_i\mathbb{Z}$, equipped with the product topology. The dynamics of the adding machine mimic base-$p$ addition, generalizing the notion of digital counters.

Adding machines have broad applications in the theory of dynamical systems. They arise naturally in the classification of minimal systems, particularly as models of regularly recurrent behavior \cite{BS},\cite{DM}, \cite{Nia}. Moreover, they serve as important tools in the construction of symbolic representations of more complex systems, such as subshifts and Bratteli-Vershik models \cite{CMG},\cite{Gao-Li}. Their simple structure and rich topological properties lead to further applications in ergodic theory and number theory, particularly through their connections with $p$-adic integers \cite{DS}, \cite{BK}, \cite{IST}.
In this paper, we extend the classical odometer by constructing it within a symbolic space known as a zip space, which is based on two distinct sets of alphabets. Zip symbolic spaces were introduced in \cite{LM} as a framework to extend two-sided shift spaces \cite{LB}, leading to the development of zip shifts—a new form of symbolic dynamics designed to encode and analyze non-invertible maps. The existence of appropriate partitions for non-invertible systems can lead to their conjugacy or semi-conjugacy with zip shift maps, making it possible to study complex and chaotic dynamics through symbolic representations.

Nevertheless, many dynamical systems lie outside the class of those that can be effectively modeled using zip shifts. In this context, the construction and study of non-zip shift maps become both meaningful and necessary. In this work, building upon the ideas and methods introduced in the work \cite{artigo}, we construct and study a bilateral adding machine that operates on a zip space. This example illustrates a homeomorphism  that is not a zip shift map. Furthermore, we provide necessary and sufficient conditions for characterizing such transformations.


\section{Zip space and Bilateral Adding Machine}

In this section, we introduce the concept of a zip space. The zip space are the base state space defined earlier in \cite{LM} to introduce the zip shift maps and spaces. Consider two finite alphabets \( Z = \{z_1, z_2, \ldots, z_k\} \) and \( S = \{s_1, s_2, \ldots, s_n\} \) with \( k \leq n \).

\begin{definition}
We say that the surjective function \( \tau \colon S \to Z \) is a \textit{transition map}. Note that \( \tau \) is not necessarily invertible.
\end{definition}

\begin{definition}\label{df1}
Let \( \Sigma_S \) be the S-full shift space. We define the (Z, S)-full zip space, denoted \( \Sigma_{Z,S} \), where each \( y \in \Sigma_{Z,S} \) corresponds to a point \( x \in \Sigma_S \) such that
\[
y_i =
\begin{cases}
x_i \in S, & \text{if } i \geq 0 \\
\tau(x_i), & \text{if } i < 0.
\end{cases}
\]
\end{definition}

\begin{example}
Let us consider the alphabets \( Z = \{a, b\} \) and \( S = \{0, 1, 2, 3\} \). We can define a transition map \( \tau \colon S \to Z \) given by:
\[
\tau(0) = \tau(2) = a, \quad \tau(1) = \tau(3) = b.
\]
Then, if we take the sequence \( (\ldots, 2, 0, 1 ; 3, 0, 1, \ldots) \), a sequence \( y \in \Sigma_{Z,S} \) has the form:
\[
y = (\ldots, \tau(2), \tau(0), \tau(1) ; 3, 0, 1, \ldots) = (\ldots, a, a, b ; 3, 0, 1, \ldots).
\]
\end{example}

\begin{definition}
Let \( \Sigma_{Z,S} \) be the (Z, S)-full zip space. We define $\sigma_{\tau}:\Sigma_{Z,S}\to\Sigma_{Z,S}$ as follows.
\begin{eqnarray}\label{ZS}
(\sigma_{\tau}(\bar x))_i=\left\{\begin{tabular}{ll}
$x_{i+1} \,\,\,\quad\quad \text{if}\,\,i\neq -1,$ \\
$\tau(x_{0}) \quad\quad \text{if}\,\,i=-1.$
\end{tabular}\right.
\end{eqnarray}
We call this map \textit{Zip shift} and the pair $(\Sigma, \sigma_{\tau})$ is called the \textit{Zip shift Space}.
\end{definition}

We define the metric \( d \colon \Sigma_{Z,S} \times \Sigma_{Z,S} \to \mathbb{R} \) given by:
\begin{equation}\label{metric}
d(x, y) =
\begin{cases}
\frac{1}{2^{M(x,y)}}, & \text{if } x \neq y \\
0, & \text{if } x = y
\end{cases}
\end{equation}
where \( M(x, y) = \min \{ |i| \; ; \; x_i \neq y_i, \; i \in \mathbb{Z} \} \).
The set of basic cylinder sets,
\[\mathcal{C}=\{C_{x_{i}}^{i}:x_{i}\in Z\,\textit{for }\, i<0\,\textit{and}\, x_{i}\in S\,\textit{for }\, i\geq 0 \}\] is a sub-basis for the product topology equivalent to the topology induced by the above metric on the (Z, S)-full zip space. The set of general cylinder sets, which generate a basis for this topology, is denotes as following. 
\[\mathcal{B}=\{C_{x_{-n}, \ldots, x_0, \ldots, x_k}^{-n,\dots,-1,0,\dots, k}: x_{i}\in Z\,\textit{for }\, i<0\,\textit{and}\, x_{i}\in S\,\textit{for }\,i\geq 0\}.\]

\subsection*{Bilateral Adding Machine}

In this subsection, we present the concept of a bilateral adding machine. Initially, let us consider the finite alphabets \( Z = \{a_0, b_1\} \) and \( S = \{0, 1, \ldots, j - 1\} \) with \( j \geq 2 \).

\begin{definition}\label{df2}
Let \( \alpha = (\ldots, 2, 2 ; j, j, \ldots) \) be a bi-infinite sequence. Consider $Z=\{a_0,b_1\}$ and $S=\{0,1,\dots,j-1\}$ and let  \( \Sigma_{Z,S} \) be the (Z, S)-full zip shift space with some \( \tau \colon S \to Z \).  We call the set \( \Delta_\alpha =\Sigma_{Z,S} \), an \textit{$\alpha$-adic bilateral adding machine space}.

\end{definition}

\begin{remark}
Since our goal is to generalize the one-sided adding machine, from the binary perspective, we set \( a_0 = 0 \) and \( b_1 = 1 \) and the indices for $a$ and $b$ are stated with this aim. Thus, \( a_0 + 1 = b_1 \) and \( b_1 + 1 = a_0 \).
\end{remark}

\begin{definition}\label{5.2.2}
Given sequences \( x, y, z \in \Delta_\alpha \), the sum \( x + y = z \) is defined as follows:

\begin{itemize}
\item For \( i \geq 0 \): The resulting coordinates are:
\[
z_0 = (x_0 + y_0) \mod j, \quad z_1 = (x_1 + y_1 + t_1) \mod j
\]
where
\[
t_1 =
\begin{cases}
0, & \text{if } x_0 + y_0 < j \\
1, & \text{if } x_0 + y_0 \geq j
\end{cases}
\]
and
\[
z_2 = (x_2 + y_2 + t_2) \mod j
\]
where
\[
t_2 =
\begin{cases}
0, & \text{if } x_1 + y_1 + t_1 < j \\
1, & \text{if } x_1 + y_1 + t_1 \geq j
\end{cases}
\]
To obtain the next terms, simply continue with this reasoning. In this case, the addition is performed as in the one-sided adding machine.

\item For \( i < 0 \): The resulting coordinates are:
\[
z_{-1} = (x_{-1} + y_{-1} + s_1) \mod 2 \quad \text{with } s_1 = \tau(x_0 + y_0\mod j),
\]
\[
z_{-2} = (x_{-2} + y_{-2} + s_2) \mod 2
\]
where
\[
s_2 =
\begin{cases}
a_0, & \text{if } x_{-1} + y_{-1} + s_1 < 2 \\
b_1, & \text{if } x_{-1} + y_{-1} + s_1 \geq 2
\end{cases}
\]
To obtain the previous terms, simply continue with this reasoning.
\end{itemize}
\end{definition}

\begin{example}
Let \( \alpha = (\ldots, 2, 2 ; 3, 3, \ldots) \), and the sequences
\[
x = (\ldots, b_1, a_0 ; 3, 1, \ldots), \quad y = (\ldots, b_1, b_1 ; 2, 0, \ldots) \in \Delta_\alpha,
\]
where the transition map $\tau: \{0,1,2\}\to\{a_0,b_1\} $ is defined as $\tau(2)=b_1$ and $\tau(0)=\tau(1)=a_0$Then:
\small{
\begin{eqnarray*}
x + y &=& (\ldots, (b_1 + b_1 + s_2) , (a_0 + b_1 + \tau(2))_{(\leftarrow\mod 2)} ;_{(\rightarrow\mod 3)} (3 + 2) , (1 + 0 + t_1) , \ldots)
\\
&=& (\ldots, (b_1 + b_1 + b_1), (a_0 + b_1 + b_1)_{(\leftarrow\mod 2)} ;_{(\rightarrow\mod 3)} 5, 2, \ldots)
\\
&=& (\ldots, b_1, a_0 ; 2, 2, \ldots)
\end{eqnarray*}
}
For simplicity, we have omitted repeated occurrences of mod 2 in the negative entries and mod 3 in the positive entries. To this aim is considered the notation
 $(\leftarrow\mod 2)$ and $(\rightarrow\mod 3)$.
\end{example}

\begin{definition}
Let \( f_\alpha \colon \Delta_\alpha \to \Delta_\alpha \) be defined as:
\[
f_\alpha(x) = (\ldots, x_{-2}, x_{-1} ; x_0, x_1, \ldots) + (\ldots, a_0, a_0 ; 1, 0,0,0, \ldots).
\tag{5.4}
\]
This is called the \textit{bilateral adding machine map}.
\end{definition}

\begin{example}
Let \( \alpha = (\ldots, 2, 2 ; 4, 4, \ldots) \) and \( x = (\ldots, a_0, b_1 ; 2, 1, \ldots) \in \Delta_\alpha \), where the transition map $\tau: \{0,1,2,3\}\to\{a_0,b_1\} $ is defined as $\tau(0)=\tau(3)=b_1$ and $\tau(1)=\tau(2)=a_0$ then:
\small{
\begin{eqnarray*}
f_\alpha(x) &=& (\ldots, (a_0 + a_0 + s_2), (b_1 + a_0 + \tau(3))_{(\leftarrow\mod 2)} ;_{(\rightarrow\mod 4)} (3), (1 + 0 + t_1), \ldots)
\\
&=& (\ldots, (a_0 + a_0 + b_1) , (b_1 + a_0 + b_1)_{(\leftarrow\mod 2)} ;_{(\rightarrow\mod 4)} 3, 1 , \ldots)
\\
&=&(\ldots, b_1, a_0 ; 3, 1, \ldots)
\end{eqnarray*}
}
\end{example}
To avoid overloading the upcoming proofs, from this point on we will omit the modulo \( j \) sums. However, it should be clear to the reader that all operations are performed according to Definition \ref{5.2.2}

\begin{proposition}
The bilateral adding machine map \( f_\alpha \) is a homeomorphism.
\end{proposition}

\begin{proof}
\textbf{\( f_\alpha \) is injective:}
Let \( x, y \in \Delta_{\alpha} \) be distinct such that \( f_{\alpha}(x) = f_{\alpha}(y) \). Then,
\small{
    \begin{equation*}
        \begin{split}
            &  f_{\alpha}(x) = f_{\alpha}(y)
             \Rightarrow \\
             &(\dots,z_{-2},z_{-1};x_0 + 1, x_2 + t_1, x_3 + t_2,  \dots) = (\dots,z_{-2},z_{-1};y_0 + 1, y_2 + s_1, y_3 + s_2, \dots). 
             \end{split}
    \end{equation*}
    Which means:
    \begin{equation*}
        \begin{split}
                          \begin{cases}
                x_0 + 1 = y_0 + 1 \Rightarrow x_0=y_0\\ 
                x_1 + t_1 = y_1 + s_1 \\
                x_2 + t_2 = y_2 + s_2 \\
                \hspace{1.5cm} \vdots  
            \end{cases}  
        \end{split}
    \end{equation*}
    and note that $\,z_{-1}=x_{-1}+\tau(x_0+1) = y_{-1}+\tau(y_0+1)$.
}
    By definition,
    \[
        \begin{cases}
            t_1 = 0, \text{ if } x_0 + 1 < j_1 \\ 
            t_1 = 1, \text{ otherwise}
        \end{cases}
        \quad \text{and} \quad
        \begin{cases}
            s_1 = 0, \text{ if } y_0 + 1 < j_1 \\ 
            s_1 = 1, \text{ otherwise}
        \end{cases}
    \]

    Since \( x_0= y_0 \), the two systems above are equal. Hence, \( t_1 = s_1 \) and \( x_1 = y_1 \), also $x_{-1}=y_{-1}$. Continuing this line of reasoning, we get \( x_i = y_i \) for all \( i \in \mathbb{Z} \). Therefore, \( f_{\alpha} \) is injective.\\
\textbf{\( f_\alpha \) is surjective:}
    Let \( y = (\ldots, y_{-2}, y_{-1} ; y_0, y_1, \ldots) \in \Delta_\alpha \). By definition, \( y_i \in \{a_0, b_1\} \) for all \( i < 0 \). We consider the following cases:

    \begin{itemize}
        \item[(a)] If \( y_{-1} = a_0 \):
        since \( y_0 \leq j - 1 \), there exists \( x = (\ldots, y_{-3}, y_{-2}, a_0 ; y_0 - 1, y_1, y_2, \ldots) \) such that:
        \begin{eqnarray*}
        f_\alpha(x) &=&  (\ldots, y_{-3}, y_{-2}, a_0 ; y_0 - 1, y_1, y_2, \ldots) + (\ldots, a_0, a_0 ; 1, 0, \ldots)
        \\
        &=& (\ldots, y_{-3} + s_3, y_{-2} + s_2, a_0 + \tau(y_0) ; y_0 - 1 + 1, y_1 + t_1, y_2 + t_2, \ldots)
        \\
       &=&  (\ldots, y_{-3}, y_{-2}, a_0 ; y_0, y_1, y_2, \ldots) = y.
       \end{eqnarray*}
        \item[(b)] If \( y_{-1} = b_1 \):
       for \( y_0 \leq j_0 - 1 \), then there exists \( x = (\ldots, y_{-3}, y_{-2}, b_1 ; y_0 - 1, y_1, y_2, \ldots) \) such that:
        \begin{eqnarray*}
        f_\alpha(x) &=& (\ldots, y_{-3}, y_{-2}, b_1 ; y_0 - 1, y_1, y_2, \ldots) + (\ldots, a_0, a_0 ; 1, 0, \ldots)
        \\
        &=& (\ldots, y_{-3} + s_3, y_{-2} + s_2, b_1 + \tau(y_0) ; y_0 - 1 + 1, y_1 + t_1, y_2 + t_2, \ldots)
        \\
        &=& (\ldots, y_{-3}, y_{-2}, b_1 ; y_0, y_1, y_2, \ldots) = y.
        \end{eqnarray*}
    \end{itemize}

    Since we have covered both possible cases, we conclude that the bilateral adding machine function is surjective.\\
\textbf{\( f_\alpha \) is bi-continuous:}
We show that the preimage of a basic cylinder set in $\Delta_{\alpha^+}$ is open in $\Delta_{\alpha}$. 
$f_{\alpha}^{-1}\left(C^{i}_{s_i}\right)$ is a cylinder. We simply observe that
\small{
    \begin{equation*}
    f_{\alpha}^{-1}\left(C^{i}_{s_i}\right) = 
        \begin{cases}
            C^{i}_{s_i}, & 
            \text{if $i\geq 0$ and $x_i$s does not simultaneously reach its maximum} \\
            C^{0,\cdots,i}_{x^*,x^*,\cdots,x^*,s_i-1}, & \text{otherwise (with $x^*=j-1$).}
        \end{cases} 
     \end{equation*}
     }
or
\small{
    \begin{equation*}
    f_{\alpha}^{-1}\left(C^{i}_{s_i}\right) = 
        \begin{cases}
            C^{i}_{s_i}, & 
            \text{if $i< 0$ and $x_i$s does not simultaneously reach its maximum} \\
            \cup_{x^*} C^{i,\cdots,-1,0}_{s_i-1,b_1,\cdots,b_1,x^*} & \text{otherwise (with $x^*\in\tau^{-1}(b_1)$).}
        \end{cases}
     \end{equation*}
     }
     That is, given $x \in f_{\alpha}^{-1}\left(C_{i}^{s_i}\right)$ such that $f_{\alpha}(x)$ does not carry over to the $i$-th coordinate, we take $x$ such that $x_i = s_i$. If $f_{\alpha}(x)$ does carry over to the $i$-th coordinate, we simply take $x_i = s_i - 1$. 
    Therefore, the preimage of a basic cylinder set is also a cylinder set and $f_{\alpha}$ is continuous. As the bilateral adding machine space is compact and \( f_\alpha \) is continuous, the inverse function \( f_\alpha^{-1} \) is also continuous.


 
\end{proof}


\begin{remark}\label{R2}
Whenever there is no risk of confusion, we may use the notation $C_{x_{-n}, \ldots, x_0, \ldots, x_k}$ instead of $C_{x_{-n}, \ldots, x_0, \ldots, x_k}^{-n,\dots,-1,0,\dots, k}$. Moreover, for any fixed natural number \( n \) and any fixed non-negative integer \( k \) such that, \( x_i \in S \) for \( i \geq 0 \) and \( x_i \in Z=\{a_0, b_1\} \) for \( i < 0 \), one has:
\[
\Delta_\alpha = \bigcup_{x_{i\geq 0}\in S,\,x_{i<0}\in Z} C_{x_{-n}, \ldots, x_0, \ldots, x_k}.
\]
\end{remark}

\begin{proposition}
The space \( \Delta_\alpha \) is minimal for \( f_\alpha \).
\end{proposition}

\begin{proof}
To simplify the exposition, we will omit the carry operations that occur during the iterations of the function so as not to overload the equations.
We want to show that \( \Delta_\alpha \) is \( f_\alpha \)-minimal. Since \( f_\alpha \) is a homeomorphism, by Proposition 2.3.1, it suffices to show:
\[
\forall x \in \Delta_\alpha, \quad \mathcal{O}_{f_\alpha}(x) \cap C_{y_{-k}, \ldots, y_{-1}, x_0, \ldots, x_k} \neq \emptyset,
\]
where \( y_i \in \{a_0, b_1\} \) for \( i < 0 \) and \( x_i \in S \) for \( i \geq 0 \).
We will prove this using mathematical induction.




%
\textbf{Case $n=2$:}
By Remark \ref{R2}, we have:
\[
\Delta_\alpha = C_{a_0, x_0}^{-1, 0} \cup C_{b_1, x_0}^{-1, 0}
\]
where \( x_0 \in \{0, 1, \ldots, j - 1\} \). Now suppose \( x \in C_{a_0, 0}^{-1, 0} \). Then,
\[
x \in C_{a_0, 0}^{-1, 0} \Rightarrow x = (\ldots, y_{-2}, a_0 ; 0, x_1, \ldots)
\]
\[
\Rightarrow f_\alpha(x) = (\ldots, y_{-2}, a_0 ; 0, x_1, \ldots) + (\ldots, a_0, a_0 ; 1, 0, \ldots)
\]
\[
\Rightarrow f_\alpha(x) = (\ldots, y_{-2} + s_2, a_0 ; 1, x_1 + t_1, \ldots)
\Rightarrow f_\alpha(x) \in C_{a_0, 1}^{-1, 0}.
\]

Similarly,
\[
f_\alpha(x) \in C_{a_0, 1}^{-1, 0} \Rightarrow f_\alpha(x) = (\ldots, y_{-2}, a_0 ; 1, x_1, \ldots)
\]
\[
\Rightarrow f_\alpha^2(x) = (\ldots, y_{-2}, a_0 ; 1, x_1, \ldots) + (\ldots, a_0, a_0 ; 1, 0, \ldots)
\]
\[
= (\ldots, y_{-2} + s_2, a_0 ; 2, x_1 + t_1, \ldots) \in C_{a_0, 2}^{-1, 0}.
\]

If we continue this process, we observe that the iterates of \( x \) under \( f_\alpha \) intersect all cylinders of size 2 that describe the space.

\textbf{Case \( n=2k + 1 \):} Suppose the hypothesis holds for cylinders of size \( 2k + 1 \), that is:
\[
\forall x \in \Delta_\alpha, \quad \mathcal{O}_{f_\alpha}(x) \cap C_{y_{-k}, \ldots, y_{-1}, x_0, \ldots, x_k} \neq \emptyset.
\]

Since the intersection above is nonempty, there exists \( m \in \mathbb{N} \) such that:
\[
f_\alpha^m(x) \in C_{y_{-k}, \ldots, y_{-1}, x_0, \ldots, x_k}.
\]

We aim to find \( n \in \mathbb{N} \) such that \( f_\alpha^n(f_\alpha^m(x)) \) belongs to some cylinder of size \( 2k + 2 \). We know that:
\[
C_{y_{-k}, \ldots, y_{-1}, x_0, \ldots, x_k} = C_{a_0, y_{-k}, \ldots, y_{-1}, x_0, \ldots, x_k}^{-k-1, \ldots, -1, 0, \ldots, k} \cup C_{b_1, y_{-k}, \ldots, y_{-1}, x_0, \ldots, x_k}^{-k-1, \ldots, -1, 0, \ldots, k}.
\]

One can also describe the space \( \Delta_\alpha \) using cylinders of size \( 2k + 2 \) where coordinate \( k+1 \) belongs to \( S \), and in this case, the argument matches the analogous result for the one-sided adding machine.

Without any loss of generality assume that:
\[
f_\alpha^m(x) \in C_{a_0, y_{-k}, \ldots, y_{-1}, x_0, \ldots, x_k}^{-k-1, \ldots, -1, 0, \ldots, k},
\]
then:
\small{ 
\begin{eqnarray}
f_\alpha^{m+1}(x) &=& (\ldots, a_0, y_{-k}, \ldots, y_{-2}, y_{-1} ; x_0, x_1, \ldots, x_k, \ldots) + (\ldots, 0, 0 ; 1, 0, \ldots)\\
&=& (\ldots, a_0 + s_{k-1}, y_{-k} + s_k, \ldots, y_{-2} + s_2, y_{-1} + s_1 ; x_0 + 1, x_1 + t_1, \ldots)
\end{eqnarray}
}
Proceeding with these iterations, we obtain:
\[
f_\alpha^j(f_\alpha^m(x)) = (\ldots, a_0 + s_{k-1} + \ldots, y_{-k} + \ldots, \ldots, b_1 + y_{-1} + \ldots ; x_0, 1 + x_1 + \ldots, x_k + \ldots, \ldots)
\]

Thus, to carry over 1 to coordinate \( -1 \) of the sequence, it suffices to iterate \( j \)-times. Similarly, to carry over 1 to coordinate \( -2 \), it is enough to iterate \( 2j \)-times:
\[
f_\alpha^{j \cdot 2}(f_\alpha^m(x)) = (\ldots, a_0 + s_{k-1} + \ldots, \ldots, b_1 + y_{-2} + \ldots, y_{-1} + \ldots ; x_0, \ldots, 2 + x_1 + \ldots, 1 + x_2 + \ldots, \ldots)
\]

To carry over 1 to coordinate \( -3 \), we iterate \( j \cdot 2^2 \)-times:
\[
f_\alpha^{j \cdot 2^2}(f_\alpha^m(x)) = (\ldots, a_0 + s_{k-1} + \ldots, \ldots, b_1 + y_{-3} + \ldots, y_{-2} + \ldots, y_{-1} + \ldots ; x_0 + \ldots, 1 + x_1 + \ldots, \ldots)
\]


Therefore, to carry over 1 to coordinate \( -k - 1 \), by the equation above, it suffices to iterate exactly \( n = 2^k j \)-times. Without loss of generality, we can choose all carryovers to be zero. Thus,
\[
(f_\alpha^m)^n(x) \in C_{1, y_{-k}, \ldots, y_{-1}, x_0, \ldots, x_k}^{-k-1, \ldots, -1, 0, \ldots, k}
\]
and the intersection with the orbit of \( f_\alpha^m(x) \) is nonempty. Therefore, \( \Delta_\alpha \) is a minimal set.
\end{proof}


In \cite{LMW}, the authors introduce the notion of \emph{S-expansivity} and is shown that zip shift maps are S-expansive local homeomorphisms. This definition extends the classical concept of two-sided expansivity to n-to-1 non-invertible maps and when the map is a homeomorphism, the definition coincides with the standard two-sided expansivity. 

\begin{definition}[\textbf{S-expansivity}]
Let $(X,d)$ be a compact metric space and let $f:X \to X$ be a continuous map. We say that $\gamma > 0$ is an \emph{S-expansivity constant} for $f$ (or that $f$ is \emph{S-expansive}) if for any pair of points $x, y \in X$, there exists an integer $n \in \mathbb{Z}$ such that
\[
d(f^n(x), f^n(y)) > \gamma.
\]
If $n$ is negative, we define $d(f^n(x), f^n(y)) := d([f^n]^{-1}(x), [f^n]^{-1}(y))$, where $[f^n]^{-1}(x)$ denotes an appropriate preimage under $f^{-n}$.
\end{definition}
\begin{proposition}\cite{LMW} \label{ze}
	The n-to-1 zip shift maps are S-expansive local homeomorphisms.
\end{proposition}
\begin{proof}
Let $x\neq y $ then there exists some $i\in \mathbb{Z}$ such that $x_{i}\neq y_{i}$. Let $i$ be the least in modulus of such $i$. If $i>0$, then $n=i+1$ and  $d(\sigma_{\tau}^{n}x,\sigma_{\tau}^{n}y)=\frac{1}{2^0}>1/2$. If $i<0$, then $d(\sigma_{\tau}^{-i}x,\sigma_{\tau}^{-i}y)=d([\sigma_{\tau}^{i}]^{-1}x,[\sigma_{\tau}^{i}]^{-1}y)=\frac{1}{2^0}>1/2$. 
Indeed the zip shift map is an S-expansive dynamical system with $\gamma=1/2$.
\end{proof}
\begin{proposition}
The bilateral adding machine map is not expansive.
\begin{proof}
We show that the adding machine \( f: \Delta_\alpha \to \Delta_\alpha \) is not expansive. Recall that a homeomorphism \( f \) on a metric space \( (X, d) \) is said to be \emph{expansive} if there exists \( \delta > 0 \) such that for all \( x \neq y \in X \), there exists \( n \in \mathbb{Z} \) such that $
d(f^n(x), f^n(y)) > \delta.$
We will show that no such \( \delta \) can exist for the bilateral adding machine. Let \( d \) be the standard metric on \( \Delta_\alpha \), as defined in
\eqref{metric}.
Let \( \delta >0\). 
Then there exists \( N \in \mathbb{N} \) that \( \frac{1}{2^N} < \delta \) and $x\neq y\in \Delta_\alpha$, such that \( d(x, y) = \frac{1}{2^N} < \delta \). Observe that under the adding machine map \( f \), the sequences \( f^n(x) \) and \( f^n(y) \) differ in digits beyond the \( N \)-th place for all \( n \). Since the action of \( f \) only carries forward and both \( x \) and \( y \) agree up to position \( N-1 \), the distance between \( f^n(x) \) and \( f^n(y) \) remains less than \( \delta \) for all \( n \in \mathbb{Z} \). Which contradicts the definition of expansiveness. Hence, the bilateral adding machine is not expansive.
\end{proof}
\end{proposition}

\begin{proposition}\label{prop:2}
Let \( A, B \subset X \) be nonempty subsets. If both \( A \) and \( B \) are \( f \)-minimal, then either \( A = B \) or \( A \cap B = \emptyset \).
\end{proposition}

\begin{proof}
Suppose \( A \cap B \neq \emptyset \). Since both \( A \) and \( B \) are \( f \)-minimal, we have
\[
f^{-1}(A \cap B) = f^{-1}(A) \cap f^{-1}(B) = A \cap B.
\]
Note that \( A \cap B \) is closed, \( f \)-invariant, and is simultaneously contained in both \( A \) and \( B \). This contradicts the minimality of \( A \) and \( B \), unless \( A = B \). Therefore, either \( A = B \) or \( A \cap B = \emptyset \).
\end{proof}

\begin{lemma} \label{le1}
Let \( X \) be a compact metric space, \( f: X \rightarrow X \) a continuous map, and suppose \( X \) is \( f \)-minimal. Let \( n \) be a positive integer. Then, for some positive integer \( t \leq n \) and some set \( M \subset X \) that is \( f^n \)-minimal, we have:
\begin{enumerate}
    \item \( X \) is the disjoint union of \( M, f(M), \dots, f^{t-1}(M) \);
    \item Each of the sets \( M, f(M), \dots, f^{t-1}(M) \) is clopen;
    \item The collection \( \{ M, f(M), \dots, f^{t-1}(M) \} \) is the collection of all subsets of \( X \) that are both \( f^t \)-minimal and \( f^n \)-minimal.
\end{enumerate}
Moreover, for each \( x \in X \), \( \overline{\mathcal{O}_{f^n}(x)} \) is an \( f^n \)-minimal set.
\begin{proof} 
The proof of (1) will be provided through different claims.
    \begin{itemize}
        \item
        \textbf{Claim 1:} There exists \( t \leq n \) such that \( f^t(M) \cap M \neq \emptyset \).\\
        Assume otherwise, then:
        \[
        f^t(M) \cap M = \emptyset, \quad \forall t \in \{1,2,\dots,n\}.
        \]
       Which means:
        \[
        f(M) \cap M = \emptyset, \quad f^2(M) \cap M = \emptyset, \quad \dots, \quad f^n(M) \cap M = \emptyset.
        \]
       This contradicts the fact that \( M \) is \( f^n \)-minimal.
        
        \item
        \textbf{Claim 2:} For any \( i < j < t \), \( f^i(M) \cap f^j(M) = \emptyset \).\\
        Suppose otherwise: let \( y \in f^i(M) \cap f^j(M) \). Then:
        \[
        \exists x, z \in M \text{ such that } f^i(x) = y = f^j(z) \Rightarrow f^{j-i}(x) = z \Rightarrow f^{j-i}(M) \cap M \neq \emptyset,
        \]
        which is a contradiction. Hence, the sets are pairwise disjoint.
        
        \item
        \textbf{Claim 3:} If \( M \) is \( f^n \)-minimal, then \( f^i(M) \) is \( f^n \)-minimal for any \( i \in \mathbb{Z} \). \\
        Note that:
        \[
        f^n(f^i(M)) = f^i(f^n(M)) = f^i(M),
        \]
        and \( f^i(M) \) is closed since \( f^i \) is continuous and \( X \) is compact. Suppose \( N \subset f^i(M) \) is a closed, \( f^n \)-invariant subset. Then \( f^{-i}(N) \subset M \), contradicting the minimality of \( M \). Therefore, \( f^i(M) \) is \( f^n \)-minimal. Hence, by Proposition \ref{prop:2}, since \( f^t(M) \) is \( f^n \)-minimal, we have \( f^t(M) = M \).
        
        \item
        \textbf{Claim 4:} The set \( A = M \cup f(M) \cup \dots \cup f^{t-1}(M) \) is \( f \)-minimal. \\
        We show \( A \) is \( f \)-invariant. Let \( x \in f^{-1}(A) \). If \( x \in f^{-1}(M) \), then:
        \[
        f(x) \in M \Rightarrow x \in f^{n-1}(M).
        \]
        If \( n - 1 < t \), the claim follows. If \( n - 1 > t \), then \( n - 1 = m + kt \) for some \( m < t \), so:
        \[
        x \in f^{n-1}(M) = f^{m+kt}(M) = f^m(f^{kt}(M)) = f^m(M),
        \]
        hence \( x \in A \). For other cases, inclusion is straightforward.
        Now, take \( x \in A \). If \( x \in f^{t-1}(M) \), then \( f(x) \in f^t(M) = M \) by Claim 3. Other cases are similar. Thus, \( f^{-1}(A) = A \). Note that, since \( A \) is a finite union of closed sets, it is closed. Also, it has no proper closed \( f \)-invariant subset, because it is a union of pairwise disjoint minimal sets. Hence, \( A \) is \( f \)-minimal.
    \end{itemize}

    Finally, since both \( A \) and \( X \) are \( f \)-minimal, by Proposition \ref{prop:2}:
    \[
    X = M \dot{\cup} f(M) \dot{\cup} \dots \dot{\cup} f^{t-1}(M).
    \]
    
Let see the proof of (2). For each \( f^i(M) \subset X \) with \( 1 \leq i \leq t \), Claim 3 shows it is closed. Since the complement is closed as well, \( f^i(M) \) is also open.

To see the proof of item (3), note that since the sets are pairwise disjoint and \( X \) is \( f \)-minimal, each set in the union is both \( f^t \)-minimal and \( f^i \)-minimal.
Moreover, let \( x \in X \). By items 1 and 3, there exists \( f^i(M) \subset X \) with \( 1 \leq i \leq t_1 \) {\color{red} $1 \leq i \leq t$} such that \( x \in f^i(M) \). Then \( \mathcal{O}_{f^n}(x) \subset f^i(M) \). Since \( f^i(M) \) is \( f^n \)-minimal, we have \( \overline{\mathcal{O}_{f^n}(x)} = f^i(M) \).
\end{proof}
\end{lemma}

\begin{corollary}
Let \( X \) be a compact metric space, \( f: X \rightarrow X \) a continuous map, and suppose \( X \) is \( f \)-minimal and \( n > 1 \) an integer. If there exists a proper closed subset \( M \subset X \) that is \( f^n \)-minimal but not \( f^t \)-minimal for any integer \( t < n \), then there exists a continuous map \( \pi: X \rightarrow \mathbb{Z}_n \) such that:
\[
(\pi \circ f)(x) = (\pi(x) + 1) \mod n.
\]
\begin{proof}
From the previous lemma, \( X = M \dot{\cup} f(M) \dot{\cup} \dots \dot{\cup} f^{n-1}(M) \).
Define \( \pi: X \rightarrow \mathbb{Z}_n \) by \( \pi(x) = i \) if \( x \in f^i(M) \) for \( 1 \leq i \leq n \). Given an open set \( U \subset \mathbb{Z}_n \), write \( U = \{i_1, i_2, \dots, i_k\} \) with \( |U| \leq n \). Then:
\[
\pi^{-1}(U) = f^{i_1}(M) \dot{\cup} f^{i_2}(M) \dot{\cup} \dots \dot{\cup} f^{i_k}(M),
\]
which is open since each \( f^i(M) \) is clopen. Thus, \( \pi \) is continuous and satisfies:
\[
(\pi \circ f)(x) = (\pi(x) + 1) \mod n.
\]
\end{proof}
\end{corollary}

\begin{theorem}\label{te1}
    Let $\alpha = (\ldots, 2, 2 ; j, j, \ldots)$ be a sequence of positive integers with $j \geq 2$, and define $m(-i, k) := 2^i \cdot j^{k+1}$ for each $i \in \mathbb{N}$ and $k \in \mathbb{Z}^+$. Let $X$ be a compact metric space and $f \colon X \to X$ a continuous map. Then $f$ is topologically conjugate to $f_\alpha$ if and only if the following statements hold:

\begin{enumerate}
    \item For each pair of integers $(-i, k)$, there exists a cover $\mathcal{P}_{-i,k}$ of $X$ consisting of $m(-i,k)$ non-empty, pairwise disjoint, clopen sets that are cyclically permuted by $f$;
    
    \item For each pair $(-i, k)$, $\mathcal{P}_{-i+n,k+m}$ with $n \in \mathbb{N}$ partitions $\mathcal{P}_{-i+n-1,k+m-1}$;
    
    \item If $W_{-i-1,k+1} \supset W_{-i-2,k+2} \supset \ldots$ is a nested sequence with $W_{-i,k} \in \mathcal{P}_{i,k}$ for each $(-i, k)$, then 
    \[
    \bigcap_{i \in \mathbb{N},\, k \in \mathbb{Z}^+} W_{-i,k}
    \]
    consists of a single point.
\end{enumerate}
\end{theorem}
\begin{proof}
($\Rightarrow$) 
Suppose the map $f$ is conjugate to the map $f_{\alpha}$. We will show that items (1), (2), and (3) hold for $\Delta_{\alpha}$, and we will use the conjugacy to infer the equivalent properties for $X$.

Indeed, since the collection of all cylinders forms a basis for the space $\Delta_{\alpha}$, it is natural to define the coverings in this space using cylinders. We define the coverings using disjoint cylinders of size $2$ that include both negative and positive coordinates (for more details, see Remark \ref{Rm3}). In this way, we obtain the following covering:
\[
Q_{(-1,0)} = \left\{ C_{-1,0}^{s_{-1},s_0} \, ; \, s_{-1} \in \{a_0,b_1 \},\ s_0 \in \{ 0,1,\dots,j-1\} \right\}
\]
with exactly $m_{(-1,0)} = 2 \cdot j$ elements. For cylinders of size $3$, we obtain:
\[
Q_{(-2,0)} = \left\{ C_{-2,-1,0}^{s_{-2},s_{-1},s_0} \, ; \, s_{-2},s_{-1} \in \{a_0,b_1 \},\ s_0 \in \{ 0,1,\dots,j-1\} \right\}
\]
with exactly $m_{(-2,0)} = 2^2 \cdot j$ elements. 

In general, for each pair $(-i,k)$, which defines cylinders of size $i+k+1$, we obtain:
\[
Q_{(-i,k)} = \left\{ C_{-i,\dots,0,\dots,k}^{s_{-i},\dots,s_0,\dots,s_{k}} \, ; \, s_i \in \{0,1,\dots,j-1\} \text{ for } i \geq 0,\ s_i \in \{a_0,b_1\} \text{ for } i < 0 \right\}
\]
as a covering for $\Delta_{\alpha}$ with $m_{(-i,k)} = 2^i \cdot j^{k+1}$ elements.

We observe:

\begin{enumerate}
    \item 
        \begin{enumerate}
            \item The coverings defined in $\Delta_{\alpha}$ were constructed so that, given a cylinder of size $2i$, the covering $Q_{(-i,k)}$ contains all cylinders that can be formed by permutations of $i$ elements from the set $\{a_0,b_1\}$ and $k$ elements from the set $\{0,1,\dots,j-1\}$ with total $m_{(-i,k)}$ possible permutations.
            \item Note that the cylinders are non-empty, pairwise disjoint, and clopen;
            \item Given a cylinder $C_{-i,\dots,0,\dots,k}^{s_{-i},\dots,s_0,\dots,s_{k}}$ in $Q_{(-i,k)}$, we have:
            \[
            f_{\alpha}^{m_i}\left(C_{-i,\dots,0,\dots,k}^{s_{-i},\dots,s_0,\dots,s_{k}}\right) = C_{-i,\dots,0,\dots,k}^{s_{-i},\dots,s_0,\dots,s_{k}}.
            \]
        \end{enumerate}
    \item Given a cylinder 
    \[
    C_{-i-1,\dots,k+1}^{s_{-i-1},\dots,s_{k+1}} \in Q_{(-i-1,k+1)},
    \]
    there exists a cylinder 
    \[
    C_{-i,\dots,k}^{s_{-i},\dots,s_{k}} \in Q_{(-i,k)}
    \]
    such that:
    \[
    C_{-i-1,\dots,k+1}^{s_{-i-1},\dots,s_{k+1}} \subset C_{-i,\dots,k}^{s_{-i},\dots,s_{k}}.
    \]
    
    \item Consider a nested sequence:
    \[
    C_{-1,0}^{s_{-1},s_0} \supset C_{-2,-1,0,1}^{s_{-2},s_{-1},s_0,s_1} \supset C_{-3,-2,-1,0,1,2}^{s_{-3}, s_{-2}, s_{-1}, s_0, s_1, s_2} \supset \dots
    \]
    with each $C_{-i,\dots,k}^{s_{-i},\dots,s_{k}}$ in $Q_{(-i,k)}$. As the diameter (i.e. $\mathrm{diam}(A) = \sup \{ d(x, y) \mid x, y \in A \}$) of a cylinder  of size $2i$ is $\frac{1}{2^{2i}}$, when $i \rightarrow \infty$, it goes to $0$. By construction, $(\Delta_{\alpha},d)$ is a complete metric space, and it follows that:
    \[
    \bigcap_{i \in \mathbb{N}} C_{-i,\dots,0,\dots,k}^{s_{-i},\dots,s_0,\dots,s_{k}} = \{ x \},
    \]
    with $x \in \Delta_{\alpha}$.
\end{enumerate}
Thus, $Q_{(-i,k)}$ satisfies items (1), (2), and (3) of the theorem. By hypothesis, $f$ is topologically conjugate to $f_{\alpha}$, so there exists a homeomorphism $h: X \rightarrow \Delta_{\alpha}$ such that $h \circ f = f_{\alpha} \circ h$. Therefore, we can construct a covering:
\[
P_{(-i,k)} = \left\{ X_{-i,\dots,0,\dots,k}^{l_{-i},\dots,l_0,\dots,l_k} \, ; \, l_r \in \{a_0,b_1\} \text{ for } r<0,\ l_r \in \{0,1,\dots,j-1\} \text{ for } r \geq 0 \right\}
\]
of $X$, where $X_{-i,\dots,0,\dots,k}^{l_{-i},\dots,l_0,\dots,l_k} = h^{-1}(C_{-i,\dots,0,\dots,k}^{s_{-i},\dots,s_0,\dots,s_k})$ and the notation $X_{-i,\dots,0,\dots,k}^{l_{-i},\dots,l_0,\dots,l_k}$ is used solely to emphasize its correspondence with cylinder sets. 

\begin{enumerate}
    \item 
        \begin{enumerate}
            \item Since $h$ is a homeomorphism, it preserves cardinality. Thus, $|P_{(-i,k)}| = m_{(-i,k)}$;
            
            \item Being a homeomorphism, $h$ maps clopen sets to clopen sets. Hence, $$X_{-i,\dots,0,\dots,k}^{l_{-i},\dots,l_k} = h^{-1}(C_{-i,\dots,0,\dots,k}^{s_{-i},\dots,s_k})$$ is non-empty, clopen, and any two elements in $P_{(-i,k)}$ are disjoint because the cylinders in $Q_{(-i,k)}$ are pairwise disjoint;
            
            \item Let $X_{-i,\dots,0,\dots,k}^{l_{-i},\dots,l_k} = h^{-1}(C_{-i,\dots,0,\dots,k}^{s_{-i},\dots,s_k})$ be an element of $P_{(-i,k)}$. Since $f_{\alpha}$ cyclically permutes the cylinders, there exists $n \in \mathbb{N}$ such that:
            \[
            f_{\alpha}^n(C_{-i,\dots,k}^{s_{-i},\dots,s_k}) = C_{-i,\dots,k}^{s_{-i},\dots,s_k}.
            \]
            Because $f$ and $f_{\alpha}$ are conjugate, we have:
            \small{
            \begin{align*}
               \hspace{2cm} f^n = h^{-1} \circ f_{\alpha}^n \circ h 
                &\Rightarrow\ f^n(h^{-1}(C_{-i,\dots,k}^{s_{-i},\dots,s_k})) = h^{-1}(f_{\alpha}^n(C_{-i,\dots,k}^{s_{-i},\dots,s_k})) 
                &= h^{-1}(C_{-i,\dots,k}^{s_{-i},\dots,s_k}) \\
                &\Rightarrow\ f^n(X_{-i,\dots,k}^{l_{-i},\dots,l_k}) = X_{-i,\dots,k}^{l_{-i},\dots,l_k}.
            \end{align*}
            }
            Therefore, $f$ cyclically permutes the elements of the covering.
        \end{enumerate}

    \item Since $Q_{(-i-1,k+1)}$ refines $Q_{(-i,k)}$, we have: $C_{-i-1,\dots,k+1}^{s_{-i-1},\dots,s_{k+1}} \subset C_{-i,\dots,k}^{s_{-i},\dots,s_k}.$
    Applying $h^{-1}$ gives: $X_{-i-1,\dots,k+1}^{l_{-i-1},\dots,l_{k+1}} \subset X_{-i,\dots,k}^{l_{-i},\dots,l_k},$ 
    hence $P_{(-i-1,k+1)}$ refines $P_{(-i,k)}$.

    \item Consider a nested sequence:
    \[
    X_{(-1,0)}^{l_{-1},l_0} \supset X_{(-2,-1,0,1)}^{l_{-2},l_{-1},l_0,l_1} \supset \dots
    \]
    with each $X_{-i,\dots,k}^{l_{-i},\dots,l_k} \in P_{(-i,k)}$. Then:
    \begin{align*}
        \bigcap_{i \in \mathbb{N}} X_{-i,\dots,k}^{l_{-i},\dots,l_k} 
        = \bigcap_{i \in \mathbb{N}} h^{-1}(C_{-i,\dots,k}^{s_{-i},\dots,s_k}) 
        = h^{-1} \left( \bigcap_{i \in \mathbb{N}} C_{-i,\dots,k}^{s_{-i},\dots,s_k} \right) 
        = h^{-1}(\{x\}),
    \end{align*}
    with $x \in \Delta_{\alpha}$. Since $h$ is a bijective homeomorphism, $h^{-1}(\{x\})$ consists of a single point.
\end{enumerate}
Therefore, items (1), (2), and (3) hold.
\hfill\\
$(\Leftarrow)$ 
Now, assume that for each pair of integers $(-i,k)$, $P_{(-i,k)}$ represents a covering of $X$ that satisfies items (1), (2), and (3). We use the notation:
\[
P_{(-i,k)} = \left\{ X_{-i,\dots,0,\dots,k}^{l_{-i},\dots,l_0,\dots,l_k} \, ; \, l_i \in \{a_0,b_1\} \text{ for } i<0, \, l_i \in \{0,1,\dots,j-1\} \text{ for } i \geq 0 \right\}.
\]
Observe that, by item (1), the coverings are structured such that their cardinalities characterize each cover. From the forward implication, we know that there exist coverings $Q_{(-i,k)}$ in the space $(\Delta_{\alpha},f_{\alpha})$ that satisfy items (1), (2), and (3) of the theorem. 

Thus, we define a map $h: X \rightarrow \Delta_{\alpha}$, where $\Delta_{\alpha}$ is covered by collections $Q_{(-i,k)}$ that satisfy the same conditions, and each element of $P_{(-i,k)}$ is uniquely associated with an element of $Q_{(-i,k)}$ for a fixed pair $(-i,k)$, with $|P_{(-i,k)}| = |Q_{(-i,k)}|$. 

Given $x \in X$, for each $(-i,k)$, there exists a unique element $X_{-i,\dots,0,\dots,k}^{l_{-i},\dots,l_k}$ that contains $x$. Since:
\[
\bigcap\limits_{i \in \mathbb{N}, \, k \in \mathbb{Z}_+} C_{-i,\dots,0,\dots,k}^{s_{-i},\dots,s_k} = \{y\}
\]
where $y \in \Delta_{\alpha}$ and the cylinders belong to the partition $Q_{(-i,k)}$, we define $h(x) = y$. We now prove that $h$ is a homeomorphism:
\begin{enumerate}
    \item The map $h$ is well-defined because for each $x \in X$, the intersection
    \[
    \bigcap\limits_{i \in \mathbb{N}, \, k \in \mathbb{Z}_+} C_{-i,\dots,0,\dots,k}^{s_{-i},\dots,s_k}
    \]
    contains a unique point $y$.

    \item Let 
    \[
    V = \bigcup\limits_{i \in \mathbb{N}, \, k \in \mathbb{Z}_+ } C_{-i,\dots,0,\dots,k}^{s_{-i},\dots,s_k}
    \]
    be an open set in $\Delta_{\alpha}$. Then:
    \[
    h^{-1}(V) = \bigcup\limits_{i \in \mathbb{N}, \, k \in \mathbb{Z}_+} h^{-1}(C_{-i,\dots,0,\dots,k}^{s_{-i},\dots,s_k}),
    \]
    and since the coverings of $X$ and $\Delta_{\alpha}$ satisfy item (1) and are matched in cardinality, $h^{-1}(V)$ is clopen in $(X,f)$. Therefore, $h$ is continuous.

    \item Let $x,y \in X$ with $x \ne y$. Then there exist distinct elements 
    $X_{-i,\dots,k}^{l_{-i},\dots,l_k}$ and $X_{-i,\dots,k}^{m_{-i},\dots,m_k}$ in $P_{(-i,k)}$ containing $x$ and $y$ respectively. By definition:
    \[
    h(x) \in \bigcap\limits_{i, k} C_{-i,\dots,k}^{l_{-i},\dots,l_k}, \quad h(y) \in \bigcap\limits_{i, k} C_{-i,\dots,k}^{m_{-i},\dots,m_k}.
    \]
    These are disjoint since the cylinders are from a partition $Q_{(-i,k)}$. Hence, $h(x) \ne h(y)$, so $h$ is injective.

    \item For $y \in \Delta_{\alpha}$, since 
    \[
    \bigcap\limits_{i,k} C_{-i,\dots,k}^{s_{-i},\dots,s_k} = \{y\},
    \]
    and each $C_{-i,\dots,k}$ corresponds uniquely to some $X_{-i,\dots,k}^{l_{-i},\dots,l_k} \in P_{(-i,k)}$, there exists a unique $x \in X$ such that $h(x) = y$. Hence, $h$ is surjective.

    \item Since $X$ is compact, $h$ is a homeomorphism.
\end{enumerate}

Finally, let us show that $h \circ f = f_{\alpha} \circ h$. Given $x \in X$, there exists a unique pair $(-i,k)$ such that $x \in X_{-i,\dots,k}^{l_{-i},\dots,l_k}$ in $P_{(-i,k)}$. Since $f$ cyclically permutes elements of $P_{(-i,k)}$, we have:
\[
f(X_{-i,\dots,k}^{l_{-i},\dots,l_k}) = X_{-i,\dots,k}^{l_{-i}+u_{-i+1},\dots,l_0+1,\dots,l_k+t_k} \in P_{(-i,k)}.
\]
Thus, for $x \in X_{-i,\dots,k}^{l_{-i},\dots,l_k}$, we have $f(x) \in X_{-i,\dots,k}^{l_{-i}+u_{-i+1},\dots,l_0+1,\dots,l_k+t_k}$ and so:
\[
h(f(x)) \in \bigcap\limits_{i,k} C_{-i,\dots,k}^{l_{-i}+u_{-i+1},\dots,l_0+1,\dots,l_k+t_k}.
\]
On the other hand, since $h(x) \in \bigcap_{i,k} C_{-i,\dots,k}^{l_{-i},\dots,l_k}$ and $f_{\alpha}$ cyclically permutes cylinders:
\[
f_{\alpha}(h(x)) = (\dots, l_{-i}+u_{-i+1}, \dots, l_0+1, \dots, l_k+t_k, \dots) \in \bigcap\limits_{i,k} C_{-i,\dots,k}^{l_{-i}+u_{-i+1},\dots,l_k+t_k}.
\]
Since the intersection determines a unique point, we conclude that $h \circ f = f_{\alpha} \circ h$.

\end{proof}
\begin{remark}[On the necessity of bilateral refinement]\label{Rm3}
Given a covering of $\Delta_{\alpha}$, it is crucial that the refinement of this covering occurs simultaneously in both the positive and negative directions. If the refinement is performed only on one side, item (3) of the theorem is not satisfied. Indeed, consider $Q_{(-i,k)}$ a covering of $\Delta_{\alpha}$ for some fixed $i$. Observe that
\[
C_{-i,\dots,0,\dots,k}^{s_{-i},\dots,s_0,\dots,s_k} \supset C_{-i,\dots,0,\dots,k,k+1}^{s_{-i},\dots,s_0,\dots,s_k,s_{k+1}} \supset \dots
\]
and the intersection
\[
\bigcap\limits_{k \in \mathbb{Z}_+} C_{-i,\dots,0,\dots,k}^{s_{-i},\dots,s_0,\dots,s_k}
\]
is not a single point. In fact, the sequences
\[
(\dots,a_0,s_{-i},\dots,s_0,\dots,s_k,\dots) \quad \text{and} \quad (\dots,b_1,s_{-i},\dots,s_0,\dots,s_k,\dots)
\]
both belong to the intersection above. This demonstrates that refinement on only one side (in this case, the future) is not sufficient to isolate a unique point, thus violating condition (3).

Therefore, in the context of refining a space $X$ with dynamics 
$f$, it is necessary that partitions be generated through simultaneous—or at least frequent—forward and backward iterations of a fixed, chosen partition with $m(-i,k)$ elements. This ensures that the refinement captures the full dynamical behavior of the system in both temporal directions. Such two-sided refinement is always achievable and is crucial for satisfying condition (3) of the theorem.
\end{remark}

\begin{definition}
    Let $X$ be a compact metric space and $f: X \rightarrow X$ a continuous function. If $X$ is an $f$-minimal set, we denote by $S(f)$ the set of positive integers $i$ such that for some $M \subset X$, $M$ is $f^i$-minimal but not $f^j$-minimal for any $j = 1,\dots,i-1$.
\end{definition}

\begin{definition}
    Let $f: X \rightarrow X$ be a continuous function. Given $x \in X$, we say that $x$ is \textbf{regularly recurrent} if for every neighborhood $V$ of $x$, there exists a positive integer $n$ such that for every non-negative integer $k$, $f^{nk}(x) \in V$.
\end{definition}

\begin{definition}
    For every integer $j > 1$ and each prime $p$ dividing $j$, we define $M_j(p)$ as the multiplicity of $p$ in the prime factorization of $j$. If $j=1$ or if $j > 1$ and $p$ does not divide $j$, we define $M_j(p) = 0$. If $M_j(p)$ is arbitrarily large for some prime factor $p$ of $j \in S(f)$, then we define $M_j(p) = \infty$. We also define $M(p) = \max\limits_{j \in S(f)} M_j(p)$.
\end{definition}

\begin{theorem}[Symbolic model for minimal systems with regular recurrence]
Let $X$ be a compact metric space and let $f: X \rightarrow X$ be a continuous function. If $X$ is infinite, $f$-minimal, and there exists a regularly recurrent point $x \in X$, then there exists a sequence $\alpha$ of prime numbers and a continuous, surjective map $\pi: X \rightarrow \Delta_{\alpha}$ such that
\[
\pi \circ f = f_{\alpha} \circ \pi.
\]
Moreover, if $x \in X$ is regularly recurrent, then the fiber over its image is a singleton:
\[
\pi^{-1}(\pi(x)) = \{x\}.
\]
\end{theorem}

\begin{proof}
    Before we begin the proof of this theorem, let us prove the following statements:
    \begin{enumerate}[label=\arabic*.]
        \item $S(f)$ is an infinite set: \label{af1}
        \begin{proof}
            It suffices to show that given any $k \in \Z^{*}_+$, there exists $n \in S(f)$ such that $n > k$. Indeed, let $x \in X$ be a regularly recurrent point and let $V$ be a neighborhood of $x$ such that $f^i(x) \notin \overline{V}$ for $i = 1, \dots, k$. By hypothesis, there exists $t \in \Z^{*}_+$ such that $f^{it}(x) \in V$ for all non-negative integers $i$. By Lemma \ref{le1}, $\overline{\mathcal{O}_{f^t}(x)}$ is $f^t$-minimal. Let $n$ be the smallest positive integer such that $\overline{\mathcal{O}_{f^t}(x)}$ is $f^n$-minimal, so $n \in S(f)$. Since $\overline{\mathcal{O}_{f^t}(x)} \subset \overline{V}$, it follows that $n > k$.
        \end{proof}
        
        \item If $k, n \in S(f)$ and $k$ is a multiple of $n$, then $\mathcal{Q}_k$ refines $\mathcal{Q}_n$: \label{af2}
        \begin{proof}
            By Lemma \ref{le1}, for each element of $S(f)$, a unique covering of $X$ can be constructed. Let $\mathcal{Q}_k$ be such a covering, and take $M \in \mathcal{Q}_k$ such that $z \in M$. Then, by Lemma \ref{le1}, $\overline{\mathcal{O}_{f^k}(z)}$ is an $f^k$-minimal set. Since there exists $t \in \Z$ such that $k = tn$, we have $\mathcal{O}_{f^k}(z) \subset \mathcal{O}_{f^n}(z)$, and therefore $\overline{\mathcal{O}_{f^k}(z)} \subset \overline{\mathcal{O}_{f^n}(z)}$. Since $\overline{\mathcal{O}_{f^n}(z)} \in \mathcal{Q}_n$, we conclude that $\mathcal{Q}_k$ refines $\mathcal{Q}_n$.
        \end{proof}
        
        \item Let $k \in S(f)$ and $n$ a positive integer. If $k$ is a multiple of $n$, then $n \in S(f)$: \label{af3}
        \begin{proof}
            Since $k \in S(f)$, there exists $M \subset X$ such that $M$ is $f^k$-minimal but not $f^j$-minimal for any $j = 1, \dots, k - 1$. As in Lemma \ref{le1}, we can construct the set
            \[
            M_1 = M \cup f^n(M) \cup f^{2n}(M) \cup \dots \cup f^{(t-1)n}(M).
            \]
            Observe:
            \begin{eqnarray*}
                f^n(M_1) &=& f^n(M \cup f^n(M) \cup f^{2n}(M) \cup \dots \cup f^{(t-1)n}(M)) \\
                &=& f^n(M) \cup f^{2n}(M) \cup f^{3n}(M) \cup \dots \cup M \\
                &=& M_1
            \end{eqnarray*}
            From the construction, $M_1$ is $f^n$-minimal. Moreover, since $f^i(M_1) \cap M_1 = \emptyset$ for $i = 1, \dots, n - 1$, we conclude that $n \in S(f)$.
        \end{proof}
        
        \item If $k, n \in S(f)$ and are coprime, then $kn \in S(f)$: \label{af4}
        \begin{proof}
            Since $k, n \in S(f)$, there exist subsets $M_1, M_2 \subset X$ such that $M_1$ is $f^k$-minimal but not $f^j$-minimal for $j = 1, \dots, k - 1$, and $M_2$ is $f^n$-minimal but not $f^i$-minimal for $i = 1, \dots, n - 1$. Let $x \in M_1 \cap M_2$ and define $Y = M_1 \cap M_2$. Then, $f^t(Y) \cap Y \neq \emptyset$ if and only if $t$ is a multiple of $n$, and likewise for $k$. Therefore, $f^t(Y) \cap Y = \emptyset$ for $t = 1, \dots, kn - 1$, and thus $kn \in S(f)$.
        \end{proof}
    \end{enumerate}

    Now let $n \in S(f)$. By definition, there exists a set $M \subset X$ such that $M$ is $f^n$-minimal but not $f^l$-minimal for any $l \in \{1, 2, \dots, n-1\}$. Then, by Lemma \ref{le1}, there exists $t \leq n$ such that
$$
X = M \dot{\cup} f(M) \dot{\cup} \dots \dot{\cup} f^{t-1}(M).
$$
Hence, for each $n \in S(f)$, there exists a unique covering $Q_n$ of $X$ consisting of $n$ pairwise disjoint, clopen, $f^n$-minimal sets that are cyclically permuted by $f$.

Let $p$ be a prime number. Consider $\alpha = (\dots, 2, 2\,;\,p, p, \dots)$ a sequence of prime numbers such that $M(2) = M(p) = \infty$, and define $m_i = m_{(-i,k)} = 2^i \cdot p^{k+1}$. Since the prime $p$ appears $M(p)$ times in the sequence, there exists $j \in S(f)$ such that $p$ divides $j$, and by statement \ref{af3}, we conclude that $p \in S(f)$. Then, by statement \ref{af4}, we conclude that $m_i \in S(f)$.

For each $i$, consider the covering $Q_{m_i} = \{X_{i,1}, X_{i,2}, \dots, X_{i,m_i}\}$ of $X$. By Lemma \ref{le1}, this yields item (1) of Theorem \ref{te1}, and by statement \ref{af2}, we obtain item (2) of Theorem \ref{te1}. In this way, we can construct a continuous and surjective map $\pi : X \rightarrow \Delta_{\alpha}$ as in the proof of Theorem \ref{te1} such that $\pi \circ f = f_{\alpha} \circ \pi$.

\par Finally, we prove that $\pi^{-1}(\pi(x)) = \{x\}$. Indeed, let $x \in X$ be a regularly recurrent point, and suppose there exists $y \in X$ with $y \neq x$ such that $y \in \pi^{-1}(\pi(x))$. Then, by the definition of $\pi$, there exists a nested sequence of sets $W_i \in Q_{m_i}$ for each $i$ such that
$$
x, y \in \bigcap\limits_{i \in \mathbb{N}} W_i.
$$
Let $V$ be a neighborhood of $x$ such that $y \notin \overline{V}$. Since $x$ is regularly recurrent, there exists a positive integer $n$ such that $f^{nk}(x) \in V$ for all positive integers $k$. We know that $\overline{\mathcal{O}_{f^n}(x)} \subset \overline{V}$, and by Lemma \ref{le1}, $\overline{\mathcal{O}_{f^n}(x)}$ is an $f^n$-minimal set. Again, by Lemma \ref{le1}, there exists $t \leq n$ such that $\overline{\mathcal{O}_{f^n}(x)}$ is $f^t$-minimal and $t \in S(f)$. Thus, we can choose the sequence $\alpha$ so that there exists $j$ such that $m_j = m_{(-j,k)}$ is a multiple of $t$. It follows that $x \in W_j$, and by statement \ref{af2}, we have $W_j \subset \overline{\mathcal{O}_{f^n}(x)}$. Since $y \in \bigcap\limits_{i \in \mathbb{N}} W_i$, we get $y \in \overline{\mathcal{O}_{f^n}(x)}$, a contradiction. Therefore, $\pi^{-1}(\pi(x)) = \{x\}$.

\end{proof}

Let $\sigma: \Sigma_{Z,S} \to \Sigma_{Z,S}$ denote a full zip shift map on $\Sigma_{Z,S}$ (Definitions \ref{df1}, \ref{df2}). Then we have the following Proposition which shows that bilateral odometers or adding machins are examples of non-zip shift maps.
\begin{proposition}\label{prop:4}
 Let $X \subset \Sigma_{Z,S}$ be a closed subset such that $\sigma(X) = X$. Then there does not exist a sequence $\alpha$ such that the restricted map $\sigma|_X : X \to X$ is topologically conjugate to the bilateral adding machine map $f_\alpha : \Delta_\alpha \to \Delta_\alpha$.
\end{proposition}

\begin{proof}
By Propositions \ref{ze} and \ref{prop:2}, the zip shift map is S-expansive, whereas the odometer map does not possess this property. Therefore, assuming the existence of a topological conjugacy leads to a contradiction.
\end{proof}


\subsection*{Data availability}
This study did not involve the use of any datasets. Therefore, no data are available.
\subsection*{Conflicts of interest}
The authors declare no conflicts of interest.
\subsection*{Acknowledgments} The authors would like to thanks Fundação de Amparo à Pesquisa do Estado de Minas Gerais (FAPEMIG) for partial financial supports. 
\bibliographystyle{alpha}

\end{document}